\documentclass[12pt,a4paper]{amsart}

\usepackage{amssymb}
\usepackage[final]{showkeys}
\usepackage{microtype}
\usepackage{color}
\usepackage{graphicx}
\usepackage[hmargin=3cm,vmargin={3.5cm,4cm}]{geometry}
\usepackage{bm}

\newtheorem{te}{Theorem}[section]

\numberwithin{equation}{section}

\allowdisplaybreaks

\begin{document}
	
	\title[]{A note on differential equations of logistic type}
	
	\author{G. Dattoli$^1$}
	\address{$^1$ ENEA - Frascati Research Center.}
	
	\author{R. Garra$^2$}
	\address{$^2$ Section of Mathematics, Università Telematica Internazionale}

	\begin{abstract}
		Logistic equations play a pivotal role in the study of any non linear evolution process exhibiting growth and saturation. The interest for the phenomenology, they rule, goes well beyond physical processes and cover many aspects of ecology, population growth, economy…According to such a broad range of applications, there are different forms of functions and distributions which are recognized as generalized logistics.
		Sometimes they are obtained by fitting procedures. Therefore, criteria might be needed to infer the associated  non linear differential equations, useful to guess “hidden” evolution mechanisms. In this article we analyze different forms of logistic functions and use simple means to reconstruct the differential equation they satisfy. Our study includes also differential equations containing non standard forms of derivative operators, like those of the Laguerre type.

		\smallskip
		
		\textit{Keywords:} Logistic functions, Laguerre derivatives, nonlinear differential equations, 
		
	\end{abstract}
	
	\maketitle
	
	\section{Introduction and preliminaries}
	
	The logistic function (hereafter LF) \cite{1}
	\begin{equation}\label{F}
		F(x,r|K) = f_0 \frac{e^{rx}}{1+\frac{f_0}{K}(e^{rx}-1)}
	\end{equation}
	is often used in different branches of science going from population statistics \cite{2}, Ecology \cite{3}, Laser Physics \cite{4}, Economy \cite{5} and tumor mass growth \cite{pagnutti}.
	As is well known, it is the solution of a first order non linear differential equation, which can be directly derived from eq. \eqref{F}. Even though elementary, we detail the derivation, for further convenience. Omitting the argument of the function $F$ for brevity’s sake, we solve \eqref{F} for $e^{rx}$ and set\\
	\begin{equation}\label{F1}
		\begin{cases}
      e^{rx} = \alpha \frac{F}{\frac{F}{K}-1}\\	
	\alpha = \frac{1}{f_0}\left(\frac{f_0}{K}-1\right)
	\end{cases}
    \end{equation}
	Keeping the derivative of both sides of eq. \eqref{F1} and, taking advantage from the fact that the exponential is an eigen-function of the derivative operator, we obtain
	\begin{equation}\label{F2}
	-r\frac{F}{\frac{F}{K}-1} = \frac{F'}{\left(\frac{F}{K}-1\right)^2}
	\end{equation}
	which can eventually be cast in the logistic equation canonical form
	\begin{equation}\label{F3}
		F' = r\left(1-\frac{F}{K}\right)F,
	\end{equation}
where $F' = dF/dx$.
	
	The meaning of eq. \eqref{F3} is clear, it describes a growth process, counteracted by a quadratic reduction of the gain.
	As already noted, the key note in the derivation of the last relationship, is the fact that the exponential function is an eigen-function of the derivative operator. 
	For further convenience, we write eq. \eqref{F} in the equivalent form
	
	\begin{equation}\label{F4}
		\begin{cases}
		\tilde{F}(x,r|\mu) =\frac{F}{K} = \frac{1}{1+\mu e^{-rx}}\\
		\mu= -\alpha K,
		\end{cases}
	\end{equation}
	which is less plethoric than its original form and allows straightforward generalizations, as illustrated below. 
	The LF belongs to the family of sigmoid curves \cite{6}. They are ubiquitous in any aspects of applied science and are often used to interpolate experimental data \cite{7}. The curve fittings provide analytical expressions, often deviating from the "canonical" logistic. The derivation of the differential equations satisfied by these modified forms, may, in principle, shine light on the mechanisms, which have contributed to e.g. the dynamics of a specific growth  process.
	The procedure we have just envisaged can be usefully exploited to derive the differential equations ruling the behavior of other forms of LF’s, like
	\begin{equation}\label{combi}
    \tilde{F}(x,r_1,r_2|\mu) =\frac{1}{1+\mu (e^{-r_1 x}+e^{-r_2 x})}.	
    \end{equation}
    Once applying the same outlined procedure (namely by solving for one of the exponentials)  we obtain a slightly extended version of eq. \eqref{F3}, i.e.
    \begin{equation}	\label{F7}
    \tilde{	F'} = r_1\bigg[1-(1+\Delta e^{-r_2 x})\tilde{F}\bigg]\tilde{F},
    \end{equation}
    where 
    $$\Delta = -\mu\frac{r_2-r_1}{r_1}.$$
    Interpreted as a logistic differential equation, in which the carrying capacity contains an exponential term. If e.g. the growth rates r have not the same sign the evolution process represented by eq. \eqref{F7} contains competing effects which determine a decay (see Fig. 1  and next sections for further comments).
    
    \begin{figure}  %%[H]
    	\centering
    	\includegraphics{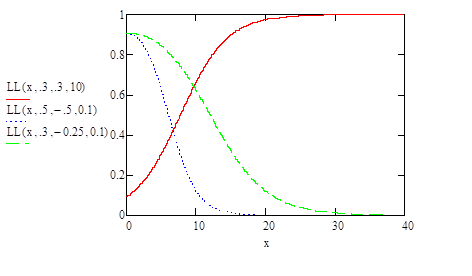}
    	\caption{Plots of the function \eqref{combi} for different values of $r_1$, $r_2$ and $\mu$. It is evident the role played by the sign of the parameters $r_1$ and $r_2$ leading to 
    	a decay or an increasing behaviour with saturation. }
    \end{figure}
    
    This article is devoted to different forms of LF and their associated differential equation, along with an interpretation, whenever possible, of their role in applications.
    
    \section{Generalized Logistic}
    
    The higher order LF is a particular case of the generalized logistic and the example we consider is the fairly simple function reported below (with $n$ not necessarily integer)
    
    \begin{equation}
    \tilde{F}(x,r|\mu) = \bigg[\frac{1}{1+\mu e^{-nrx}}\bigg]^{\frac{1}{n}}.	
    \end{equation}
	A straightforward application of the envisaged procedure, yields, for the associated differential equation, the relationship
    \begin{equation}
    	\tilde{F'} = r\left(1-\tilde{F}^n\right)\tilde{F},
    \end{equation}
     which is the Richard’s growth equation \cite{8}. The more general case (used in forestry to study the plant/forest evolution \cite{9})
     \begin{equation}
     	\sigma  = \alpha + \left(\frac{\lambda}{\chi+\eta e^{-kx}}\right)^{\frac{1}{n}}
     \end{equation}
 is characterized by a similar equation
\begin{equation}
	\sigma' = \frac{k}{n\lambda}\left[\lambda-\chi(\sigma-\alpha)^n\right](\sigma-\alpha).
\end{equation}
It is evident that to get either equations (2.2) and (2.4) we have handled (2.1) and (2.3) to enucleate the exponential and then procede as before. The same criterion can be applied if it is possible to determine, inside the generalized LF, any other function, recognized as eigen-function of a specific differential operator.\\
We consider therefore an LF of the type
\begin{equation}\label{gen}
	Z = \frac{1}{1+\mu P(x)},
\end{equation}
where $P(x)$ is the characteristic of the LF (in brief LC - Logistic Characteristic). The choice of $P(x)$ is determined by specific mathemamatical constraints associated with the specific problem under study.
For example, $P(x)$ should be sufficiently regular, positive and decreasing function. Moreover, a general construction can be suggested by taking $P(x)$ as an eigenfunction 
of a linear differential operator $\widehat{O}_x$, i.e. 
$$\widehat{O}_x P(x) = \delta P(x),$$
where $\delta$ is a real coefficient. \\
Just to give an example, we consider the function $P(x) = \cosh(rx)$ in eq. \eqref{gen}. By taking into account that the hyperbolic cosine is an eigenfunction of the second derivative, we find that the generalized logistic function
$$Z = \frac{1}{1+\mu \cosh(rx)},$$
is a solution for the logistic-type differential equation
\begin{equation}\label{cos}
	Z''-2f(Z)Z' = -r^2(1-Z)Z,
\end{equation}
where
\begin{equation}
\nonumber f(Z) = \frac{d}{dx}\ln(Z).
\end{equation}
The equation \eqref{cos} describes therefore a kind of non-linear damped 
oscillation.\\

A straightforward extension of the previous example is the case with LC 
\begin{equation}\label{expsu}
P(x) = c_1 e^{-r_1 x}+ c_2 e^{-r_2 x}
\end{equation}
which yields to the logistic function
\begin{equation}\label{expsu1}
Z = \frac{1}{1+\mu \left(c_1 e^{-r_1 x}+ c_2 e^{-r_2 x}\right)}.
\end{equation}
The LC \eqref{expsu} satisfies the linear differential equation
\begin{equation}
	\left(\frac{d^2}{dx^2}+(r_1+r_2)\frac{d}{dx}\right)P(x) = -r_1r_2P(x),
\end{equation}
i.e. it is an eigenfunction of the operator 
$$\widehat{O}_x=	\left(\frac{d^2}{dx^2}+(r_1+r_2)\frac{d}{dx}\right).$$
Then, by using the fact that 
$$\mu P(x) = \frac{1-Z}{Z}$$
we have that the $LF$ \eqref{expsu1} satisfies the equation 
\begin{equation}
	\left(\frac{d^2}{dx^2}+(r_1+r_2)\frac{d}{dx}\right)\frac{1-Z}{Z} = -r_1r_2\frac{1-Z}{Z}.
\end{equation}
Concluding, the logistic function \eqref{expsu1} is a solution for the nonlinear logistic-type equation
\begin{equation}\label{cos1}
	Z''+(r_1+r_2)Z'-2f(Z)Z' = -r_1r_2(1-Z)Z,
\end{equation}
where
\begin{equation}
	\nonumber f(Z) = \frac{d}{dx}\ln(Z).
\end{equation}
In the theory of Free Electron Laser, the dynamics of the system is ruled by an LF  whose LC is the solution of a third order ODE \cite{10}, in the final section we discuss how the simple point of view developed so far, applies to this type of problems too.

    \section{The Logistic function and Laguerre derivative}
    
    In the previous section we have envisaged a “paradigm” to derive the associated nonlinear differential equation starting from the logistic-type function (LF). The procedure consists in fixing the differential equation for the LC, which provides the “finger prints” of the linear part of the desired logistic equation.
    In this section we extend our analysis to characteristics functions which are not elementary transcendent, but e.g. higher order transcendent of the Bessel type.
    The concept of Laguerre derivative was introduced within the context of a study aimed at  establishing the properties of multivariable Laguerre polynomials \cite{11}. Successive researches \cite{14,13,12,15} have underscored the possibility of exploiting the underlying concepts in a wider context, involving generalized forms of exponentials \cite{16}.  
   	Let us denote with the symbol $_L\widehat{D}_x$ the Laguerre derivative, i.e.
	$$_L\widehat{D}_x := \frac{d}{dx} x\frac{d}{dx} .$$
Without entering further into the relevant properties, we just remind that the function
	\begin{equation}\label{tric}
		e_0(x)= \sum_{k=0}^\infty \frac{x^k}{k!^2} = I_0(2\sqrt{x}),
	\end{equation}
satisfies the eigenvalue equation (see e.g. \cite{16,17,18})
\begin{equation}\label{eig}
	_L\widehat{D}_x e_0(rx) = r e_0(rx).
\end{equation}
The function \eqref{tric} is known as Laguerre exponential function or  0-th order Tricomi function, while $I_0(\cdot)$ is the 0-th order modified Bessel function of first kind.\\
The eigenvalue equation \eqref{eig} and the analogy between $e_0(x)$ and the exponential function is suggestive of the possible existence of growth models ruled by a \textit{Laguerre derivative dynamics} (see e. g. \cite{ricci2,ricci1}). Even though no statistical evidence seems to support such a possibility, the relevant mathematical study is worth to be pursued, since it provides further tools for the handling of  the growth data.\\ 
 We observe that in the framework of these Laguerre-type Malthusian models, the stretched Laguerre exponential function $N(t)= e_0(\lambda t^\alpha)$ gives the solution for the following generalization of the Korf model (see e.g. \cite{anto})
 \begin{equation}
 \frac{d}{dt} t\frac{dN}{dt}= \lambda \alpha^2 t^{\alpha-1} N(t).
 \end{equation}
 
 Inspired by these studies, we here devolope a different approach. Here we introduce a logistic function obtained by replacting the exponential function with the Tricomi function as LC and then we find the governing equation.
 Indeed, the Laguerre exponential function $e_0(x)$ is a useful generalization of the exponential function (see \cite{ricci1}).
 Therefore, we can consider the following Laguerre-type LF
 \begin{equation}\label{llog}
 	Z = \frac{1}{1+\mu[e_0(\lambda x)]^{-1}}.
 \end{equation}
Note that, obviously, $[e_0(\lambda x)]^{-1}\neq e_0(-x)$.
A comparison between ordinary and Laguerre logistics is provided in Fig. 2. The behavior is similar; a saturation follows the growth which in the latter case, for the same parameter of the first, is smoother.

\begin{figure}  %%[H]
	\centering
	\includegraphics{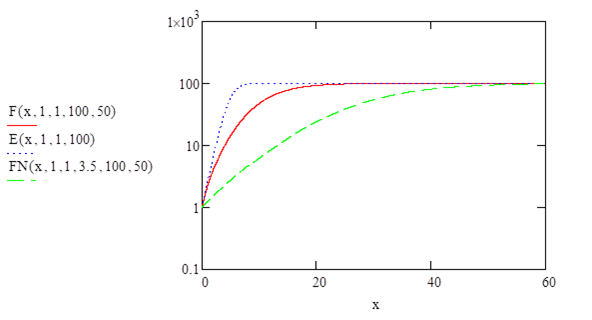}
	\caption{Here we compare the behaviour of the logistic functions:\\
	$$E(x)= \frac{a e^{\lambda x}}{1+\frac{a}{K}(e^{\lambda x}-1)}\equiv \mbox{Classical Logistic function (blue dotted line)}$$
	$$F(x)= \frac{a e_0(\lambda x)}{1+\frac{a}{K}(e_0(\lambda x)-1)}\equiv \mbox{Laguerre Logistic function (red line)}$$
	$$FN(x)= \frac{a e_\nu(\lambda x)\Gamma(\nu+1)}{1+\frac{a}{K}(e_{\nu}(\lambda x)\Gamma(\nu+1)-1)}\equiv \mbox{Laguerre Logistic with $\nu= 3.5$ (green dotted line)}$$
	We have fixed $a = \lambda = 1$ and $K = 100$.
}
\end{figure}
 
	We have the following result
	\begin{te}
The generalized logistic function \eqref{llog} satisfies the nonlinear differential equation
\begin{equation}
	_L\widehat{D}_x Z + f(Z,x) Z' = \lambda Z\left(1-Z\right),
\end{equation}
where 
$$f(Z,x) = -2x\frac{d}{dx}ln\left(1-Z\right).$$
	\end{te}

\begin{proof}
	Starting from \eqref{llog} we have that
	\begin{equation}
Z(x) \cdot \left(e_0(\lambda x)+\mu\right) = e_0(\lambda x).
	\end{equation}
  	 We apply the Laguerre derivative $_L\widehat{D}_x$ 
  	 to both sides, obtaining 
  	 \begin{equation}\label{lo2}
  	 _L\widehat{D}_x \bigg[	Z(x)\left(e_0(\lambda x)+\mu\right)\bigg]= \lambda e_0(\lambda x).	
  	 \end{equation}
    Let us introduce an auxiliary function
    $G(x) = e_0(\lambda x)+\mu$. By using the basic properties of the derivatives we have from \eqref{lo2} 
    \begin{equation}\label{lo1}
 _L\widehat{D}_x[Z(x)G(x)] = G(x) _L\widehat{D}_x Z(x)+Z(x) _L\widehat{D}_x G(x) +2x\frac{d}{dx}G\frac{d}{dx}Z= \lambda e_0(\lambda x).
    \end{equation}
We now observe that
$$_L\widehat{D}_x G = \lambda e_0(\lambda x).$$
Dividing all the terms appearing in \eqref{lo1} by $G$ we have that
\begin{equation}\label{lo3}
	_L\widehat{D}_x Z+\lambda Z^2 +2xZ'\frac{G'}{G}= \lambda Z.
\end{equation}
Finally, observing that
$$e_0(\lambda x)= \frac{\mu Z}{1-Z},$$
we obtain the claimed result by direct calculation.
	\end{proof}

More in general, if we keep into account that the $\nu$-th order Tricomi function
\begin{equation}
	e_\nu(x) = \sum_{r=0}^\infty \frac{x^r}{r!\Gamma(\nu+r+1)}, \quad \nu >-1
\end{equation}
we can construct a generalized LF replacing $e_0(\cdot)$ with $e_\nu(\cdot)$ and we can find the governing equation in a similar way, recalling that
$$_L\widehat{D}_{x,\nu}e_\nu(rx)= re_\nu(rx),$$
	where 
	$$_L\widehat{D}_{x,\nu} := \frac{d}{dx} x\frac{d}{dx} +\nu\frac{d}{dx}.$$
	
In this more general case, by analogy with the previous result, we have the following Theorem.

	\begin{te}
	The generalized logistic function 
	\begin{equation}
			Z_{\nu} = \frac{1}{1+\mu[e_\nu(\lambda x)]^{-1}}
	\end{equation}
	 satisfies the nonlinear differential equation
	\begin{equation}
		_L\widehat{D}_{x,\nu}  Z + f(Z,x) Z' = \lambda Z\left(1-Z\right),
	\end{equation}
	where 
	$$f(Z,x) = -2x\frac{d}{dx}ln\left(1-Z\right).$$
\end{te}
The proof can be obtained by calculations in analogy with the Theorem 3.1. The behaviour the generalized logistic function $Z_{\nu}$ is clearly ruled by the additive parameter $\nu$ (see Fig. 2). 
	
\section{Laguerre-type diffusion-reaction equation}

In the previous section we have introduced the Laguerre-Logistic equation, just using an analogy between the 0-th order Tricomi and the exponential function and the concept of Laguerre derivative, which also appear in a generalization of the heat equation. An interesting new topic is the analysis of nonlinear reaction-diffusion equations involving Laguerre deritivatives (see also \cite{tomo}).\\

The Laguerre diffusive equation writes 
\begin{equation}\label{ladif}
		\frac{\partial F}{\partial t} = _L\widehat{D}_x F = \frac{\partial}{\partial x}x\frac{\partial F}{\partial x}.
\end{equation}
Considering an initial value problem under the condition $F(x,0) = g(x)$, where $g(x)$ is a suitable function, a formal operatorial solution can be obtained by using a non-exponential evolution operator and it reads (see \cite{Dattoli})
\begin{equation}
F(x,t)  =e_0(t _L\widehat{D}_x)g(x).
\end{equation}
The interest for \eqref{ladif} or for relevant modified forms is justified by the role that it plays in different physical problems associated with classical optics and accelerator Physics, further comments can be found in \cite{khan,Dattoli} and reference therein. We have already underlined that eq. \eqref{ladif} is a kind of linear diffusion equation, we seek therefore for a generalization of the associated non-linear Burgers-type equation.  Accordingly we start from the Hopf-Cole transformation (see e.g. \cite{sachdev} for an appropriate discussion)
\begin{equation} \label{cole}
	u = \frac{\partial_x F}{F}
\end{equation}
and derive, with the help of \eqref{ladif}, the equation satisfied by $u$.
After keeping the (partial) time derivative of both sides of \eqref{cole} with respect to the variable $t$, we obtain
\begin{equation}\label{cole1}
	\frac{\partial u}{\partial t} = \frac{\displaystyle\left(\frac{\partial}{\partial t}\frac{\partial F}{\partial x}\right)F-\frac{\partial F}{\partial x}\frac{\partial F}{\partial t}}{F^2}.
\end{equation}
Being  independent variables, we can assume that the associated derivatives commute, so that, also on account of eq.\eqref{ladif}, we find
\begin{equation}
\frac{\partial}{\partial t}\frac{\partial F}{\partial x} = \frac{\partial}{\partial x}\left(\frac{\partial}{\partial x}x\frac{\partial F}{\partial x}\right)
\end{equation}
	Thus, going back to \eqref{cole1} and by using the fact that
	$\partial_x F = u F$, we obtain 
	\begin{equation}
			\frac{\partial u}{\partial t} = \frac{1}{F}\left(\frac{\partial}{\partial x}\left(\frac{\partial}{\partial x}(xuF)-u_L\frac{\partial}{\partial x}x\frac{\partial F}{\partial x} \right) \right).
	\end{equation}
Working out the derivatives we eventually end up with
\begin{equation}
	\frac{\partial u}{\partial t}=\frac{\partial}{\partial x}x\frac{\partial u}{\partial x}+\frac{\partial u}{\partial x}+\left(1+x\frac{\partial}{\partial x}\right) u^2
\end{equation}
	which is a non-trivial non-linear reaction-diffusion type equation with known solution, once the explicit solution of the linear Laguerre-type equation \eqref{ladif} is given.
	This useful scheme permits for example to construct interesting explicit solutions for nonlinear partial differential equations by using Laguerre polynomials (as well as in the classical Burgers equation by using the Hermite polyonomials, see \cite{levi}).
	This point will be object of further research.\\
	
	Observe that another interesting generalization of the Burgers equation is given by 
	\begin{equation}
		\frac{\partial u}{\partial t} = \frac{\partial}{\partial x}x\frac{\partial u}{\partial x} +x	\left(\frac{\partial u}{\partial x}\right)^2,
	\end{equation}
that can be linearized and reduced to the equation \eqref{ladif} by means of the transformation $u = \ln F$.

	\section{Final comments}
	
	Before going further it is worth to underline that being (1.4) and (2.2) Riccati-type equations, they can be reduced to linear forms, as indicated below. Regarding eq. (1.4), by taking
	$ F = 1/E$ and $k = r/K$, we obtain the linear equation
	\begin{equation}
	E' = -rE+k.
	\end{equation}
	The initial problem has been accordingly reduced to an elementary first order non homogeneous differential equation, with known analytical solutions also for explicitly time dependent growth rate and carrying capacity. In this case, the use of standard means yields
	\begin{equation}
		E = e^{-\int_0^x r(t')dt'} \left(E_0+\int_0^x e^{\int_0^{t'} r(t'')dt''}k(t')dt'\right),
	\end{equation}
	which, for the non linear counterpart, provides the result
	\begin{equation}
	F = \frac{f_0 e^{\int_0^{x} r(t')dt'}}{1+f_0\int_0^x e^{\int_0^{t'} r(t'')dt''}k(t')dt'},	
	\end{equation}
 where $r(t)$ and $k(t)$ are suitable positive values functions.
 An analogous result has been mentioned in ref. \cite{16}.  \\
 A similar procedure can be used to obtain the linear counterpart of  eq. (2.2). By taking $\tilde{F}= 1/E^n$ and introducing $Y = E^n$, we get indeed
 \begin{equation}
 	Y' = -\frac{r}{n}(Y-1).
 \end{equation}
 In the previous sections we have discussed logistic type functions and we have looked for a general criterion to determine the NL differential equation they satisfy. This point of view can be extended to LF exhibiting an explicit time dependence. 
 We have considered so far equations containing one variable only and therefore we have left out other forms of noticable importance in applications, like logistic-diffusive equations of the Fisher type \cite{16}. They model growth-saturation processes including also heat type diffusion contributions and write
 	\begin{equation}\label{rea}
 		\frac{\partial f}{\partial t}-\lambda \frac{\partial^2 f}{\partial x^2} = G(f),
 	\end{equation}
 where $G(f) = f(1-f^n)$. This is the diffusion-reaction equation corresponding to the previous model considering both space and time variables. 
 The relevant solutions are usually written in the travelling wave form, which can be worded as it follows.
 The solution of the non linear diffusive equation \eqref{rea} reads (see \cite{18a,18,18b})
 \begin{equation}
 	f(x,t) = A\tilde{F}^2 (\xi),
 \end{equation}
where $\xi = k(x-Vt)$ and $\tilde{F}$ is given by the solution of the equation (2.2). Furthermore $A$ is a constant to be determined, along with $k$ and $V$ are the wave vector and velocity respectively, whose dispersion relation is fixed by the structure of eq. \eqref{rea}. For this type of problems the complication arises from the fact that the initial constant and the other specifying the trial solution are entangled by non linear relationships \cite{17,18}.  
Regarding e.g. the case $G(u) = \mu u(1-u)$, the travelling wave solution reads 
\begin{equation}
	f(x,t) = \frac{1}{k^2}\left(\frac{1}{1+e^{-\xi}}\right)^2,
\end{equation}
where
\begin{equation}
	k = \pm \sqrt{\frac{\mu}{6\alpha}}, \quad V = \pm 5\alpha k. 
\end{equation}
The usefulness of logistic type functions has been recently been underscored for the modeling of high gain Free Electron Laser (FEL) devices \cite{19}. Within this context it has been suggested that the function modeling the evolution of the laser field complex amplitude is
\begin{equation}\label{fel}
	l(\tau) = l_0\frac{a(\tau)}{1+\frac{l_0}{|l_F|}(a(\tau)-1)},
\end{equation}
with $a(\tau)$ satisfying the third order ODE \cite{10}
\begin{equation}\label{FEL}
\widehat{F}_L a(\tau) =\left(\partial_{\tau}^3+2i\nu\partial^2_{\tau}-\nu^2\partial_{\tau}\right)a(\tau)= i\pi g_0 a(\tau),
\end{equation}
under the initial conditions
\begin{equation}\nonumber
a|_{\tau= 0} = 1, \quad \partial_{\tau}a|_{\tau= 0}= \partial_{\tau}a|_{\tau= 0}=0.
\end{equation}
It is worthless to comment here on the relevant physical meaning.
The importance of the logistic equation for the FEL theory has been recently discussed in \cite{19} and the phenomenological consequences of eq. \eqref{FEL} will be drawn elsewhere. 
The associated non linear equation is obtained by a slight extension of the previously outline procedure. We first write eq. \eqref{fel} as
\begin{equation}\label{fel3}
	a(\tau)  = A\frac{\tilde{l}(\tau)}{1-\tilde{l}(\tau)},
\end{equation}
where $\tilde{l}(\tau)= l(\tau)/l_F$ and $A = \frac{1}{\tilde{l}_0}-1$.
Then, by applyng the operator $\widehat{F}_L$ to both the sides of \eqref{fel3}
we obtain that
\begin{equation}
\tilde{l}'''+\tilde{l}''(6f+2i\nu)+(1-\tilde{l})(6f^3+4i\nu f^2-\nu^2 f-i\pi g_0 \tilde{l})=0,
\end{equation}
where 
$$f \equiv f(l) = -\frac{d}{d\tau}\ln(\tilde{l}-1).$$
The non linear equation we have derived is, apparently, quite cumbersome. However it contains some interesting element allowing further speculations on the FEL saturation mechanisms.\\
In this article we have discussed a few interesting problems whose study are worth to be pursued. In a forthcoming note we will discuss how the present analysis can be extended to logistic equations including extravariables and how the nonlinear Laguerre reaction-diffusion can be framed in a wider context.

\end{document}